\documentclass[reqno,b5paper]{amsart}

\usepackage{amsmath}
\usepackage{amssymb}
\usepackage{amsthm}
\usepackage{enumerate}
\usepackage[mathscr]{eucal}
\usepackage{eqlist}

\theoremstyle{plain}
\newtheorem{theorem}{Theorem}[section]

\newtheorem{lemma}[theorem]{Lemma}

\newtheorem{corollary}[theorem]{Corollary}
\theoremstyle{definition}
\newtheorem{definition}[theorem]{Definition}

\newtheorem{remark}[theorem]{Remark}
\numberwithin{equation}{section}
\usepackage{enumerate}
\numberwithin{equation}{section}
\begin{document}
\title[A new algorithm for mixed equilibrium problem]%
{A new algorithm for mixed equilibrium problem and Bregman strongly nonexpansive mappings in Banach spaces}
\author[Vahid Darvish]%
{Vahid Darvish}

\newcommand{\acr}{\newline\indent}
\address{Department of Mathematics and Computer Science\\
Amirkabir University of Technology\\
Hafez Ave., P.O. Box 15875-4413\\Tehran, Iran.} \email{vahid.darvish@mail.com}

\subjclass[2010]{47H05, 47J25, 58C30}
\keywords{Banach space, Bregman projection, Bregman distance, Bregman strongly nonexpansive mapping, fixed point, mixed equilibrium problem. }

\begin{abstract}
In this paper, we study a new iterative method for a common fixed point of a finite family of Bregman strongly nonexpansive mappings in the frame work of reflexive real Banach spaces. Moreover, we prove the strong convergence theorem for finding common fixed points with the solutions of a mixed equilibrium problem.
\end{abstract}

\maketitle

\section{Introduction}
Let $E$ be a real reflexive Banach space and $C$ a nonempty, closed and convex subset of $E$ and  $E^{*}$ be the dual space of $E$ and $f:E\to (-\infty,+\infty]$ be a proper, lower semi-continuous and convex function. We denote by $\text{dom} f$, the domain of $f$, that is the set $\{x\in E : f(x)<+\infty\}$. Let $x\in \text{int}(\text{dom} f)$, the subdifferential  of $f$ at $x$ is the convex set defined by 
\begin{equation*}
\partial f(x)=\{x^{*}\in E^{*} : f(x)+\langle x^{*},y-x\rangle \leq f(y), \forall y\in E\},
\end{equation*}
where the Fenchel conjugate of $f$ is the function $f^{*}: E^{*}\to (-\infty,+\infty]$ defined by 
$$f^{*}(x^{*})=\sup \{\langle x^{*},x\rangle -f(x): x\in E\}.$$

Equilibrium problems which were introduced by Blum and Oettli \cite{blu} and Noor and Oettli \cite{asl} in 1994 have had a great
impact and influence in the development of several branches of pure and applied sciences. It has been shown that the
equilibrium problem theory provides a novel and unified treatment of a wide class of problems which arise in economics,
finance, image reconstruction, ecology, transportation, network, elasticity and optimization. It has been shown (\cite{blu},\cite{asl}) that
equilibrium problems include variational inequalities, fixed point, Nash equilibrium and game theory as special cases. Hence
collectively, equilibrium problems cover a vast range of applications. Due to the nature of the equilibrium problems, it is
not possible to extend the projection and its variant forms for solving equilibrium problems. To overcome this drawback,
one usually uses the auxiliary principle technique. The main and basic idea in this technique is to consider an auxiliary
equilibrium problem related to the original problem and then show that the solution of the auxiliary problems is a solution
of the original problem. This technique has been used to suggest and analyze a number of iterative methods for solving
various classes of equilibrium problems and variational inequalities, see \cite{asl2}, \cite{cen}  and the references therein.
Related to the equilibrium problems, we also have the problem of finding the fixed points of the nonexpansive mappings,
which is the subject of current interest in functional analysis. It is natural to construct a unified approach for these problems.
In this direction, several authors have introduced some iterative schemes for finding a common element of a set of the
solutions of the equilibrium problems and a set of the fixed points of finitely many nonexpansive mappings, see \cite{yao} and the references therein.

Let $\varphi :C\to \mathbb{R}$ be a real-valued function and $\Theta: C\times C\to \mathbb{R}$ be an equilibrium bifunction. The mixed equilibrium problem (for short, MEP) is to find $x^{*}\in C$ such that
$$\text{MEP}: \Theta (x^{*},y)+\varphi(y)\geq \varphi(x^{*}), \ \ \forall y\in C.$$
In particular, if $\varphi \equiv 0$, this problem reduces to the equilibrium problem (for short, EP), which is to find $x^{*}\in C$ such that 
$$\text{EP}: \Theta (x^{*}, y)\geq 0, \ \ \forall y\in C.$$
The mixed equilibrium problems include fixed point problems, optimization problems, variational inequality problems, Nash equilibrium problems and the equilibrium problems as special cases; see for example \cite{blu}, \cite{cha},\cite{cha2} and \cite{kon}.

In \cite{rei}, Reich and Sabach proposed an algorithm for finding a common fixed point of finitely many Bregman strongly nonexpansive mappings $T_{i}:C\to C (i=1,2,\ldots, N)$ satisfying $\cap_{i=1}^{N}F(T_{i})\neq \emptyset$ in a reflexive Banach space $E$ as follows:
\begin{eqnarray*}
x_{0}&\in & E, \text{chosen arbitrarily,}\\
y_{n}^{i}&=&T_{i}(x_{n}+e_{n}^{i}),\\
C_{n}^{i}&=&\{z\in E : D_{f}(z,y_{n}^{i})\leq D_{f}(z,x_{n}+e_{n}^{i})\},\\
C_{n}&=&\cap_{i=1}^{N}C_{n}^{i},\\
Q_{n}^{i}&=&\{z\in E : \langle \nabla f(x_{0})-\nabla f(x_{n}), z-x_{n}\rangle\leq 0\},\\
x_{n+1}&=&proj_{C_{n}\cap Q_{n}}^{f}(x_{0}), \ \ \forall n\geq0,
\end{eqnarray*}
and
\begin{eqnarray*}
x_{0}&\in & E,\\
C_{0}^{i}&=&E, i=1,2,\ldots,N,\\
y_{n}^{i}&=&T_{i}(\nu_{n}+e_{n}^{i}),\\
C_{n+1}^{i}&=&\{z\in C_{n}^{i} : D_{f}(z,y_{n}^{i})\leq D_{f}(z,x_{n}+e_{n}^{i})\},\\
C_{n+1}&=&\cap_{i=1}^{N}C_{n+1}^{i},\\
x_{n+1}&=&proj_{C_{n+1}}(x_{0}), \ \ \forall n\geq0,
\end{eqnarray*}
where $proj_{C}^{f}$ is the Bregman projection with respect to $f$ from E onto  a closed and convex subset $C$ of $E$. They proved that the sequence $\{x_{n}\}$ converges strongly to a common fixed point of $\{T_{i}\}_{i=1}^{N}$.

In \cite{sua}, Suantai, et al used the following Halpern's iterative scheme for Bregman strongly nonexpansive self mapping $T$ on $E$; for $x_{1}\in E$ let $\{x_{n}\}$ be a sequence defined by
$$x_{n+1}=\nabla f^{*}(\alpha_{n}\nabla f(u)+(1-\alpha_{n})\nabla f(Tx_{n})), \ \ \forall n\geq1,$$
where $\{\alpha_{n}\}$ satisfying $\lim_{n\to\infty}\alpha_{n}=0$ and $\sum_{n=1}^{\infty}\alpha_{n}=\infty$. They proved that above sequence converges strongly to a fixed point of $T$.

In \cite{zeg}, Zegeye presented the following iterative scheme:
$$x_{n+1}=Proj_{C}^{f}\nabla f^{*}(\alpha_{n}\nabla f(u)+(1-\alpha_{n})\nabla f(Tx_{n}),$$
where $T=T_{N}\circ T_{N-1}\circ \ldots\circ T_{1}$. He proved that above sequence converges strongly to a common fixed point of a finite family of Bregman strongly nonexpansive mappings on a nonempty, closed and convex subset $C$ of $E$.

The authors of \cite{kum} introduced the following algorithm:
\begin{eqnarray}
x_{1}&=&x\in C \ \ \ \ \  \text{chosen arbitrarily},\nonumber\\
z_{n}&=&Res_{H}^{f}(x_{n}),\nonumber\\
y_{n}&=&\nabla f^{*}(\beta_{n}\nabla f(x_{n})+(1-\beta_{n})\nabla f(T_{n}(z_{n})))\nonumber\\
x_{n+1}&=&\nabla f^{*}(\alpha_{n}\nabla f(x_{n})+(1-\alpha_{n})\nabla f(T_{n}(y_{n}))),\label{mnb}
\end{eqnarray}
where $H$ is an equilibrium bifunction and $T_{n}$ is a Bregman strongly nonexpansive mapping for any  $n\in \mathbb{N}$. They proved the sequence (\ref{mnb})  converges strongly to the point $proj_{F(T)\cap EP(H)}x$.

In this paper, motivated by above algorithms, we study  the following iterative scheme:
\begin{eqnarray}
x_{1}&=&x\in C \ \ \ \ \  \text{chosen arbitrarily},\nonumber\\
z_{n}&=&Res_{\Theta,\varphi}^{f}(x_{n}),\nonumber\\
y_{n}&=&proj_{C}^{f}\nabla f^{*}(\beta_{n}\nabla f(x_{n})+(1-\beta_{n})\nabla f(T(z_{n})))\nonumber\\
x_{n+1}&=&proj_{C}^{f}\nabla f^{*}(\alpha_{n}\nabla f(x_{n})+(1-\alpha_{n})\nabla f (T(y_{n}))),\label{eqw}
\end{eqnarray}
where $\varphi :C\to \mathbb{R}$ is a real-valued function,  $\Theta: C\times C\to \mathbb{R}$ is an equilibrium bifunction and $T=T_{N}\circ T_{N-1}\circ \ldots\circ T_{1}$ which $T_{i}$ is a finite family of Bregman strongly nonexpansive mapping for each $i\in \{1,2,\ldots, N\}$. 
We will prove that the sequence $\{x_{n}\}$ defined in (\ref{eqw}) converges strongly to the point $proj_{(\cap_{i=1}^{N}F(T_{i}))\cap MEP(\Theta)}x$.
\section{Preliminaries}
For any $x\in \text{int}(\text{dom} f)$, the right-hand derivative of $f$ at $x$ in the derivation $y\in E$ is defined by
$$f^{'}(x,y):=\lim _{t\searrow0} \frac{f(x+ty)-f(x)}{t}.$$
The function $f$ is called G\^{a}teaux differentiable at $x$ if $\lim_{t\searrow0} \frac{f(x+ty)-f(x)}{t}$ exists for all $y\in E$. In this case, $f^{'}(x,y)$ coincides with $\nabla f(x)$, the value of the gradient ($\nabla f)$ of $f$ at $x$. The function $f$ is called G\^{a}teaux differentiable if it is G\^{a}teaux differentiable for any $x\in \text{int}(\text{dom} f)$ and $f$ is called Fr\'{e}chet differentiable at $x$ if this limit is attain uniformly for all $y$ which satisfies $\|y\|=1$. The function $f$ is uniformly Fr\'{e}chet differentiable on a subset $C$ of $E$ if the limit is attained uniformly for any $x\in C$ and $\|y\|=1$. It is known that if $f$ is G\^{a}teaux differentiable (resp. Fr\'{e}chet differentiable) on $\text{int}(\text{dom} f)$, then $f$ is continuous and its G\^{a}teaux derivative $\nabla f$ is norm-to-weak$^*$ continuous (resp. continuous) on $\text{int} (\text{dom}f)$ (see \cite{bon}).

Let $f: E\to (-\infty,+\infty]$ be a G\^{a}teaux differentiable function. The function $D_{f}: \text{dom} f\times \text{int}(\text{dom} f)\to [0,+\infty)$ defined as follows:
\begin{equation}\label{1}
D_{f}(x,y):=f(x)-f(y)-\langle \nabla f(y),x-y\rangle
\end{equation}
is called the Bregman distance with respect to $f$, \cite{cens}.

The Legendre function $f:E\to (-\infty,+\infty]$ is defined in \cite{bau}. It is well known that in reflexive spaces, $f$ is Legendre function if and only if it satisfies the following conditions:

($L_{1}$) The interior of the domain of $f$, $\text{int}(\text{dom} f)$, is nonempty, $f$ is G\^{a}teaux differentiable on $\text{int}(\text{dom} f)$ and $\text{dom} f=\text{int}( \text{dom} f)$;

($L_{2}$) The interior of the domain of $f^{*}$, $\text{int}( \text{dom} f^{*})$, is nonempty, $f^{*}$ is G\^{a}teaux differentiable on $\text{int}(\text {dom} f^{*})$ and $\text{dom} f^{*}= \text{int}( \text{dom} f^{*})$.

\noindent Since $E$ is reflexive, we know that $(\partial f)^{-1}=\partial f^{*}$ (see \cite{bon}). This , with ($L_{1}$) and ($L_{2}$), imply the following equalities: 
$$ \nabla f=(\nabla f^{*})^{-1},  \ \ \ \text{ran} \nabla f=\text{dom} \nabla f^{*}=\text{int}(\text{dom} f^{*})$$
 and $$\text {ran} \nabla f^{*}=\text{dom}(\nabla f)=\text {int}(\text{dom} f),$$ where $\text{ran}\nabla f$ denotes the range of $\nabla f$. 

When the subdifferential of $f$ is single-valued, it coincides with the gradient $\partial f=\nabla f$, \cite{phe}. By Bauschke
et al \cite{bau} the conditions ($L_{1}$) and ($L_{2}$) also yields that the function $f$ and $f^{*}$ are strictly convex on the interior of their respective domains.\\
If $E$ is a smooth and strictly convex Banach space, then an important and interesting Legendre function is $f(x):=\frac{1}{p}\|x\|^{p} (1<p<\infty).$ In this case the gradient $\nabla f$ of $f$ coincides with the generalized duality mapping of $E$, i.e., $\nabla f=J_{p} (1<p<\infty).$ In particular, $\nabla f=I$, the identity mapping in Hilbert spaces. From now on we assume that the convex function $f:E\to (-\infty, \infty]$ is Legendre.
\begin{definition}
Let $f:E\to (-\infty,+\infty]$ be a convex and G\^{a}teaux differentiable function. The Bregman projection of $x\in \text{int}(\text{dom} f)$ onto the nonempty, closed and convex subset $C\subset \text{dom} f$ is the necessary unique vector $proj_{C}^{f}(x)\in C$ satisfying 
$$D_{f}(proj_{C}^{f}(x),x)=\inf\{D_{f}(y,x) : y\in C\}.$$
\end{definition}
\begin{remark}
If $E$ is a smooth and strictly convex Banach space and $f(x)=\|x\|^{2}$ for all $x\in E$, then we have that $\nabla f(x)=2Jx$ for all $x\in E$, where $J$ is the normalized duality mapping from $E$ in to $2^{E^{*}}$, and hence $D_{f}(x,y)$ reduced to $\phi(x,y)=\|x\|^{2}-2\langle x,Jy\rangle +\|y\|^{2}$, for all $x,y\in E$, which is the Lyapunov function introduced by Alber \cite{alb} and Bregman projection $proj_{C}^{f}(x)$ reduces to the generalized projection $\Pi_{C}(x)$ which is defined by
$$\phi(\Pi_{C}(x),x)=\min_{ y\in C} \phi(y,x).$$
If $E=H$, a Hilbert space, $J$ is  the identity mapping and hence Bregman projection $proj_{C}^{f}(x)$ reduced to the metric projection of $H$ onto $C$, $P_{C}(x)$.
\end{remark}

\begin{definition}\cite{but2}
Let $f:E\to(-\infty,+\infty]$ be a convex and G\^{a}teaux differentiable function. $f$ is called:
\begin{enumerate}
\item
\textit{totally convex} at $x\in \text{int}(\text{dom} f)$ if its modulus of total convexity at $x$, that is, the function $\nu_{f}:\text{int}(\text{dom} f)\times[0,+\infty)\to[0,+\infty)$ defined by 
$$\nu_{f}(x,t):=\inf\{D_{f}(y,x): y\in \text{dom}f, \|y-x\|=t\},$$
is positive whenever $t>0$;
\item
totally convex if it is totally convex at every point $x\in \text{int}(\text{dom} f)$;
\item
totally convex on bounded sets if $\nu_{f}(B,t)$ is positive for any nonempty bounded subset $B$ of $E$ and $t>0$, where the modulus of total convexity of the function $f$ on the set $B$ is the function $\nu_{f}:\text{int}(\text{dom} f)\times [0,+\infty)\to [0,+\infty)$ defined by 
$$\nu_{f}(B,t):=\inf\{\nu_{f}(x,t): x\in B\cap \text{dom} f\}.$$
\end{enumerate}
\end{definition}
The set $lev_{\leq}^{f}(r)=\{x\in E: f (x)\leq r\}$ for some $r\in\mathbb{R}$ is called a sublevel of $f$.
\begin{definition}\cite{but2,rei}
The function $f:E\to(-\infty,+\infty]$ is called;
\begin{enumerate}
\item
\textit{cofinite} if $\text{dom} f^{*}=E^{*}$;
\item
\textit{coercive} \cite{hir} if the sublevel set of $f$ is bounded; equivalently,
$$\lim_{\|x\|\to+\infty}f(x)=+\infty;$$
\item
\textit{strongly coercive} if $\lim_{\|x\|\to+\infty}\frac{f(x)}{\|x\|}=+\infty$;
\item
\textit{sequentially consistent} if for any two sequences $\{x_{n}\}$ and $\{y_{n}\}$ in $E$ such that $\{x_{n}\}$ is bounded,
$$\lim_{n\to\infty} D_{f}(y_{n},x_{n})=0\Rightarrow \lim_{n\to\infty}\|y_{n}-x_{n}\|=0.$$
\end{enumerate}
\end{definition}
\begin{lemma}\cite{but}\label{lem6}
The function $f$ is totally convex on bounded subsets if and only if it is sequentially consistent.
\end{lemma}

\begin{lemma}\cite[Proposition 2.3]{rei}
If $f:E\to(-\infty,+\infty]$ is Fr\'{e}chet differentiable and totally convex, then $f$ is cofinite.
\end{lemma}
\begin{lemma}\cite{but}\label{jad}
Let $f:E\to(-\infty,+\infty]$ be a convex function whose domain contains at least two points.Then the following statements hold:
\begin{enumerate}
\item
$f$ is sequentially consistent if and only if it is totally convex on bounded sets;

\item
If $f$ is lower semicontinuous, then $f$ is sequentially consistent if and only if it is uniformly convex on bounded sets;

\item
If $f$ is uniformly strictly convex on bounded sets, then it is sequentially consistent and the converse implication holds when $f$ is lower semicontinuous, Fr\'{e}chet differentiable on its domain and Fr\'{e}chet derivative $\nabla f$ is uniformly continuous on bounded sets.

\end{enumerate}

\end{lemma}
\begin{lemma}\cite[Proposition 2.1]{rei4}\label{lem7}
Let $f:E\to\mathbb{R}$ be uniformly Fr\'{e}chet differentiable and bounded on bounded subsets of $E$. Then $\nabla f$ is uniformly continuous on bounded subsets of $E$ from the strong topology of $E$ to the strong topology of $E^{*}$.
\end{lemma}
\begin{lemma}\cite[Lemma 3.1]{rei}\label{taz}
Let $f:E\to \mathbb{R}$ be a G\^{a}teaux differentiable and totally convex function. If $x_{0}\in E$ and the sequence $\{D_{f}(x_{n},x_{0})\}$ is bounded, then the sequence $\{x_{n}\}$ is also bounded.

\end{lemma}

Let $T:C\to C$ be a nonlinear mapping. The fixed points set of $T$ is denoted by $F(T)$, that is $F(T)=\{x\in C: Tx=x\}$. A mapping $T$ is said to be nonexpansive if $\|Tx-Ty\|\leq \|x-y\|$ for all $x,y\in C$. $T$ is said to be quasi-nonexpansive if $F(T)\neq \emptyset$ and $\|Tx-p\|\leq \|x-p\|,$ for all $x\in C$ and $p\in F(T)$. A point $p\in C$ is called an asymptotic fixed point of $T$ (see \cite{rei2}) if $C$ contains a sequence $\{x_{n}\}$ which converges weakly to $p$ such that $\lim_{n\to\infty}\|x_{n}-Tx_{n}\|=0$. We denote by $\widehat{F}(T)$ the set of asymptotic fixed points of $T$.

A mapping $T:C\to\text{int}(\text{dom} f)$ with $F(T)\neq\emptyset$ is called:
\begin{enumerate}
\item quasi-Bregman nonexpansive \cite{rei} with respect to $f$ if 
$$D_{f}(p,Tx)\leq D_{f}(p,x), \forall x\in C, p\in F(T).$$
\item
Bregman relatively nonexpansive \cite{rei} with respect to $f$ if,
$$ D_{f}(p,Tx)\leq D_{f}(p,x), \ \ \forall x\in C, p\in F(T), \ \ \ \text{and} \ \ \widehat{F}(T)=F(T).  $$
\item
Bregman strongly nonexpansive (see \cite{bru,rei}) with respect to $f$ and $\widehat{F}(T)$ if,
$$ D_{f}(p,Tx)\leq D_{f}(p,x), \ \ \forall x\in C, p\in \widehat{F}(T) $$
and, if whenever $\{x_{n}\}\subset C$ is bounded, $p\in \widehat{F}(T)$, and 
$$\lim_{z\to\infty}(D_{f}(p,x_{n})-D_{f}(p,Tx_{n}))=0,$$
it follows that 
$$\lim_{n\to\infty} D_{f}(x_{n},Tx_{n})=0.$$
\item
Bregman firmly nonexpansive (for short BFNE) with respect to $f$ if, for all $x,y\in C$,
$$\langle \nabla f(Tx)-\nabla f(Ty),Tx-Ty\rangle \leq \langle \nabla f(x)-\nabla f(y),Tx-Ty\rangle$$
equivalently,
\begin{equation}\label{5}
D_{f}(Tx,Ty)+D_{f}(Ty,Tx)+D_{f}(Tx,x)+D_{f}(Ty,y)\leq D_{f}(Tx,y)+D_{f}(Ty,x).
\end{equation}
\end{enumerate}
The existence and approximation of Bregman firmly nonexpansive mappings was studied in \cite{rei2}. It is also known that if $T$ is Bregman firmly nonexpansive and $f$ is Legendre function which is bounded, uniformly Fr\'{e}chet differentiable and totally convex on bounded subset of $E$, then $F(T)=\widehat{F}(T)$ and $F(T)$ is closed and convex. It also follows that every Bregman firmly nonexpansive mapping is Bregman strongly nonexpansive with respect to $F(T)=\widehat{F}(T)$.

\begin{lemma}\label{niaz}\cite{but}
Let $C$ be a nonempty, closed  and convex subset of $E$. Let $f:E\to \mathbb{R}$ be a G\^{a}teaux differentiable and totally convex function. Let $x\in E$, then \\
1) $z=proj_{C}^{f}(x)$ if and only if 
$$\langle \nabla f(x)-\nabla f(z),y-z\rangle \leq 0, \ \ \ \forall y\in C.$$
2)
 $D_{f}(y,proj_{C}^{f}(x))+D_{f}(proj_{C}^{f}(x),x)\leq D_{f}(y,x),  \ \ \ \forall x\in E, y\in C.$
\end{lemma}

Let $f:E\to\mathbb{R}$ be a convex, Legendre and G\^{a}teaux differentiable function. Following \cite{alb} and \cite{cens}, we make use of the function $V_{f}:E\times E^{*}\to [0,\infty)$ associated with $f$, which is defined by
$$V_{f}(x,x^{*})=f(x)-\langle x^{*},x\rangle +f^{*}(x^{*}), \ \ \ \ \forall x\in E, x^{*}\in E^{*}.$$ 
Then $V_{f}$ is nonexpansive and $V_{f}(x,x^{*})=D_{f}(x,\nabla f^{*}(x^{*}))$ for all $x\in E$ and $x^{*}\in E^{*}$. Moreover, by the subdifferential inequality,
\begin{equation}\label{29}
V_{f}(x,x^{*})+\langle y^{*},\nabla f^{*}(x^{*})-x\rangle\leq V_{f}(x,x^{*}+y^{*})
\end{equation}
for all $x\in E$ and $x^{*},y^{*}\in E^{*}$ \cite{koh}. In addition, if $f:E\to (-\infty,+\infty]$ is a proper lower semicontinuous function, then $f^{*}:E^{*}\to(-\infty,+\infty]$ is a proper weak$^{*}$ lower semicontinuous and convex function (see \cite{mar}). Hence, $V_{f}$ is convex in the second variable. Thus, for all $z \in E$,
$$D_{f}\left(z,\nabla f^{*}\left(\sum_{i=1}^{N}t_{i}\nabla f(x_{i})\right)\right)\leq \sum_{i=1}^{N}t_{i}D_{f}(z,x_{i}),$$
where $\{x_{i}\}_{i=1}^{N}\subset E$ and $\{t_{i}\}_{i=1}^{N}\subset (0,1)$ with $\sum_{i=1}^{N}t_{i}=1$.
\begin{lemma}\label{222}\cite{mar}
Let $f:E\to(-\infty,+\infty]$ be a bounded, uniformly Fr\'{e}chet differentiable and totally convex function on bounded subsets of $E$. Assume that $\nabla f^{*}$ is bounded on bounded subsets of $\text{dom} f^{*}=E^{*}$ and let $C$ be a nonempty subset of $\text{int}(\text{dom} f)$. Let $\{T_{i} : i=1,2,\ldots,N\}$ be $N$ Bregman strongly nonexpansive mappings from $C$ into itself satisfying $\cap_{i=1}^{N}\widehat{F}(T_{i})\neq \emptyset.$ Let $T=T_{N}\circ T_{N-1}\circ\ldots \circ T_{1}$, then $T$ is Bregman strongly nonexpansive mapping and $\widehat{F}(T)=\cap_{i=1}^{N}\widehat{F}(T_{i})$.
\end{lemma}
\begin{lemma}\label{333}\cite{rei3}
Let $C$ be a nonempty, closed and convex subset of $\text{int}(\text{dom} f)$ and $T:C\to C$ be a quasi-Bregman nonexpansive mappings with respect to $f$. Then $F(T)$ is closed and convex.
\end{lemma}

For solving the mixed equilibrium problem, let us give the following assumptions for the bifunction $\Theta$ on the set $C$:

($A_{1}$) $\Theta (x,x)=0$ for all $x\in C$;

($A_{2}$) $\Theta$ is monotone, i.e., $\Theta (x,y)+\Theta (y,x)\leq0$ for any $x,y\in C$;

($A_{3}$) for each $y\in C, x\mapsto \Theta (x,y)$ is weakly upper semicontinuous;

($A_{4}$) for each $x\in C, y\mapsto \Theta (x,y)$ is convex;

($A_{5}$) for each $x\in C, y\mapsto \Theta (x,y)$ is lower semicontinuous (see \cite{pen}).

\begin{definition}
Let $C$ be a nonempty, closed and convex subsets of a real reflexive Banach space and let $\varphi$ be a lower semicontinuous and convex functional from $C$ to $\mathbb{R}$. Let $\Theta: C\times C\to\mathbb{R}$ be a bifunctional satisfying ($A_{1}$)-($A_{5}$). The \textit{mixed resolvent} of $\Theta$ is the operator $Res_{\Theta,\varphi}^{f}:E\to 2^{C}$ 
\begin{equation}
Res_{\Theta,\varphi}^{f}(x)=\{z\in C: \Theta (z,y)+\varphi(y)+\langle \nabla f(z)-\nabla f(x),y-z\rangle \geq \varphi(z), \ \ \forall y\in C\}.
\end{equation}
\end{definition}

In the following two lemmas the idea of proofs are the same as in \cite{rei}, but for reader's convenience we provide their proofs.
\begin{lemma}
Let $f:E\to(-\infty,+\infty]$ be a coercive and G\^{a}teaux differentiable function. Let $C$ be a closed and convex subset of $E$. Assume that $\varphi:C\to\mathbb{R}$ be a lower semicontinuous and convex functional and the bifunctional $\Theta:C\times C\to\mathbb{R}$ satisfies conditions ($A_{1}$)-($A_{5}$), then $\text{dom}(Res_{\Theta,\varphi}^{f})=E$.
\end{lemma}
\begin{proof}
Since $f$ is a coercive function, the function $h:E\times E\to (-\infty,+\infty]$ defined by
$$h(x,y)=f(y)-f(x)-\langle x^{*},y-x\rangle,$$
satisfies the following for all $x^{*}\in E^{*}$ and $y\in C$ 
$$\lim_{\|x-y\|\to+\infty}\frac{h(x,y)}{\|x-y\|}=+\infty.$$
Then from \cite[Theorem 1]{blu}, there exists $\hat{x}\in C$ such that
\begin{equation*}
\Theta(\hat{x},y)+\varphi(y)-\varphi(\hat{x})+f(y)-f(\hat{x})-\langle x^{*},y-\hat{x}\rangle\geq0,
\end{equation*}
for any $y\in C$. So, we have
\begin{equation}\label{7}
\Theta(\hat{x},y)+\varphi(y)+f(y)-f(\hat{x})-\langle x^{*},y-\hat{x}\rangle\geq\varphi(\hat{x}).
\end{equation}
We know that inequality (\ref{7}) holds for $y=t\hat{x}+(1-t)\hat{y}$ where $\hat{y}\in C$ and $t\in (0,1)$. Therefore,
\begin{eqnarray*}
\Theta(\hat{x},t\hat{x}+(1-t)\hat{y})&+&\varphi(t\hat{x}+(1-t)\hat{y})+f(t\hat{x}+(1-t)\hat{y})-f(\hat{x})\\
&&-\langle x^{*},t\hat{x}+(1-t)\hat{y}-\hat{x}\rangle\\
&&\geq\varphi(\hat{x})
\end{eqnarray*}
for all $\hat{y}\in C$. By convexity of $\varphi$ we have
\begin{eqnarray}
\Theta(\hat{x},t\hat{x}+(1-t)\hat{y})&+&(1-t)\varphi(\hat{y})+f(t\hat{x}+(1-t)\hat{y})-f(\hat{x})\nonumber\\
&&-\langle x^{*},t\hat{x}+(1-t)\hat{y}-\hat{x}\rangle\nonumber\\
&&\geq(1-t)\varphi(\hat{x}).\label{eq4}
\end{eqnarray}
Since 
$$f(t\hat{x}+(1-t)\hat{y})-f(\hat{x})\leq \langle \nabla f(t\hat{x}+(1-t)\hat{y}),t\hat{x}+(1-t)\hat{y}-\hat{x}\rangle,$$
we have from (\ref{eq4}) and ($A_{5}$) that
\begin{eqnarray*}
t\Theta (\hat{x},\hat{x})+(1-t)\Theta (\hat{x},\hat{y})+(1-t)\varphi(\hat{y})&+&\langle \nabla f(t\hat{x}+(1-t)\hat{y}),t\hat{x}+(1-t)\hat{y}-\hat{x}\rangle\\
&& -\langle x^{*},t\hat{x}+(1-t)\hat{y}-\hat{x}\rangle\geq (1-t)\varphi(\hat{x})
\end{eqnarray*}
for all $\hat{y}\in C$. From ($A_{1}$) we have
\begin{eqnarray*}
(1-t)\Theta (\hat{x},\hat{y})+(1-t)\varphi(\hat{y})&+&\langle \nabla f(t\hat{x}+(1-t)\hat{y}),(1-t)(\hat{y}-\hat{x})\rangle\\
&& -\langle x^{*},(1-t)(\hat{y}-\hat{x})\rangle\geq (1-t)\varphi(\hat{x}).
\end{eqnarray*}
Equivalently
\begin{eqnarray*}
(1-t)[\Theta (\hat{x},\hat{y})+\varphi(\hat{y})&+&\langle \nabla f(t\hat{x}+(1-t)\hat{y}),\hat{y}-\hat{x}\rangle\\
&& -\langle x^{*},\hat{y}-\hat{x}\rangle]\geq (1-t)\varphi(\hat{x}).
\end{eqnarray*}
So, we have
$$\Theta (\hat{x},\hat{y})+\varphi(\hat{y})+\langle \nabla f(t\hat{x}+(1-t)\hat{y}),\hat{y}-\hat{x}\rangle-\langle x^{*},\hat{y}-\hat{x}\rangle\geq\varphi(\hat{x}),$$
for all $\hat{y}\in C$. Since $f$ is G\^{a}teaux differentiable function, it follows that $\nabla f$ is norm-to-weak$^{*}$ continuous (see \cite[Proposition 2.8]{phe}. Hence, letting $t\to1^{-1}$ we then get
$$\Theta (\hat{x},\hat{y})+\varphi(\hat{y})+\langle \nabla f(\hat{x}),\hat{y}-\hat{x}\rangle-\langle x^{*},\hat{y}-\hat{x}\rangle\geq\varphi(\hat{x}).$$
By taking $x^{*}=\nabla f(x)$ we obtain $\hat{x}\in C$ such that
$$\Theta (\hat{x},\hat{y})+\varphi(\hat{y})+\langle \nabla f(\hat{x})-\nabla f(x),\hat{y}-\hat{x}\rangle\geq\varphi(\hat{x}),$$
for all $\hat{y}\in C$, i.e., $\hat{x}\in Res_{\Theta,\varphi}^{f}(x)$. So, $\text{dom}(Res_{\Theta,\varphi}^{f})=E$.

\end{proof}

\begin{lemma}\label{nv}
Let $f:E\to(-\infty,+\infty]$ be a Legendre function. Let $C$ be a closed and convex subset of $E$. If the bifunction $\Theta: C\times C\to\mathbb{R}$ satisfies conditions ($A_{1})$-($A_{5}$), then 
\begin{enumerate}

\item $Res_{\Theta,\varphi}^{f}$ is single-valued;

\item $Res_{\Theta,\varphi}^{f}$ is a BFNE operator;

\item $F\left(Res_{\Theta,\varphi}^{f}\right)=MEP(\Theta)$;

\item $MEP(\Theta)$ is closed and convex;

\item $D_{f}\left(p, Res_{\Theta,\varphi}^{f}(x)\right)+D_{f}\left(Res_{\Theta,\varphi}^{f}(x),x\right)\leq D_{f}(p,x), \ \forall p\in F\left(Res_{\Theta,\varphi}^{f}\right), x\in E$.
\end{enumerate}
\end{lemma}
\begin{proof}
(1) Let $z_{1}, z_{2}\in Res_{\Theta,\varphi}^{f}(x)$ then by definition of the resolvent we have

$$\Theta(z_{1},z_{2})+\varphi (z_{2}) +\langle \nabla f(z_{1})-\nabla f(x),z_{2}-z_{1}\rangle\geq \varphi(z_{1})$$
and
$$ \Theta (z_{2},z_{1})+\varphi(z_{1})+\langle \nabla f(z_{2}-\nabla f(x),z_{1}-z_{2}\rangle \geq \varphi(z_{2}).$$
Adding these two inequalities, we obtain
$$\Theta (z_{1},z_{2})+\Theta(z_{2},z_{1})+\varphi(z_{1})+\varphi(z_{2})+\langle \nabla f(z_{2})-\nabla f(z_{1}),z_{1}-z_{2}\rangle\geq \varphi(z_{1})+\varphi(z_{2}).$$
So, $$\Theta (z_{1},z_{2})+\Theta(z_{2},z_{1})+\langle \nabla f(z_{2})-\nabla f(z_{1}),z_{1}-z_{2}\rangle\geq 0.$$
By ($A_{2}$), we have
$$\langle \nabla f(z_{2})-\nabla f(z_{1}),z_{1}-z_{2}\rangle\geq0.$$
Since $f$ is Legendre then it is strictly convex. So, $\nabla f$ is strictly monotone and hence $z_{1}=z_{2}$. It follows that $Res_{\Theta,\varphi}^{f}$ is single-valued. \\
\\
(2) Let $x,y\in E$, we then have
\begin{eqnarray}
\Theta(Res_{\Theta,\varphi}^{f}(x),Res_{\Theta,\varphi}^{f}(y))&+&\varphi(Res_{\Theta,\varphi}^{f}(y)) \nonumber\\
&&+\langle \nabla f(Res_{\Theta,\varphi}^{f}(x))-\nabla f(x),Res_{\Theta,\varphi}^{f}(y)-Res_{\Theta,\varphi}^{f}(x)\rangle\nonumber\\
&&\geq \varphi(Res_{\Theta,\varphi}^{f}(x))\label{eq1}
\end{eqnarray}
and
\begin{eqnarray}
\Theta(Res_{\Theta,\varphi}^{f}(y),Res_{\Theta,\varphi}^{f}(x))&+&\varphi(Res_{\Theta,\varphi}^{f}(x))\nonumber\\
&&+\langle \nabla f(Res_{\Theta,\varphi}^{f}(y)-\nabla f(y),Res_{\Theta,\varphi}^{f}(x)-Res_{\Theta,\varphi}^{f}(y)\rangle\nonumber\\
&&\geq \varphi(Res_{\Theta,\varphi}^{f}(y)).\label{eq2}
\end{eqnarray}
Adding the inequalities (\ref{eq1}) and (\ref{eq2}), we have
\begin{eqnarray*}
&&\Theta (Res_{\Theta,\varphi}^{f}(x),Res_{\Theta,\varphi}^{f}(y))+\Theta(Res_{\Theta,\varphi}^{f}(y),Res_{\Theta,\varphi}^{f}(x))\\
&&+\langle \nabla f(Res_{\Theta,\varphi}^{f}(x))-\nabla f(x)+\nabla f(y)-\nabla f(Res_{\Theta,\varphi}^{f}(y)),Res_{\Theta,\varphi}^{f}(y)-Res_{\Theta,\varphi}^{f}(x)\rangle\geq0.
\end{eqnarray*}
By ($A_{2}$), we obtain
\begin{eqnarray*}
&&\langle \nabla f(Res_{\Theta,\varphi}^{f}(x))-\nabla f(Res_{\Theta,\varphi}^{f}(y)),Res_{\Theta,\varphi}^{f}(x)-Res_{\Theta,\varphi}^{f}(y)\rangle\\
&&\leq\langle \nabla f(x)-\nabla f(y),Res_{\Theta,\varphi}^{f}(x)-Res_{\Theta,\varphi}^{f}(y)\rangle.
\end{eqnarray*}
It means $Res_{\Theta,\varphi}^{f}$ is BFNE operator.\\
\\
(3) 
\begin{eqnarray*}
x\in F(Res_{\Theta,\varphi}^{f})&\Leftrightarrow& x=Res_{\Theta,\varphi}^{f}(x)\\
&\Leftrightarrow& \Theta (x,y)+\varphi(y)+\langle \nabla f(x)-\nabla f(x),y-x\rangle\geq \varphi(x), \ \ \forall y\in C\\
&\Leftrightarrow& \Theta (x,y)+\varphi(y)\geq \varphi(x), \ \ \forall y\in C\\
&\Leftrightarrow& x\in MEP(\Theta).\\
\end{eqnarray*}
\\
(4) Since $Res_{\Theta,\varphi}^{f}$ is a BFNE operator, it follows from \cite[Lemma 1.3.1]{rei3} that $F(Res_{\Theta,\varphi}^{f})$ is a closed and convex subset of $C$. So, from (3) we have $MEP(\Theta)=F(Res_{\Theta,\varphi}^{f})$ is a closed and convex subset of C.\\
\\
(5) Since $Res_{\Theta,\varphi}^{f}$ is a BFNE operator, we have from (\ref{5}) that for all $x,y\in E$
\begin{eqnarray*}
&&D_{f}(Res_{\Theta,\varphi}^{f}(x),Res_{\Theta,\varphi}^{f}(y))+D_{f}(Res_{\Theta,\varphi}^{f}(y),Res_{\Theta,\varphi}^{f}(x))\\
&&\leq D_{f}(Res_{\Theta,\varphi}^{f}(x),y)-D_{f}(Res_{\Theta,\varphi}^{f}(x),x)+D_{f}(Res_{\Theta,\varphi}^{f}(y),x)-D_{f}(Res_{\Theta,\varphi}^{f}(y),y).
\end{eqnarray*}
Let $y=p\in F(Res_{\Theta,\varphi}^{f})$, we then get
\begin{eqnarray*}
&&D_{f}(Res_{\Theta,\varphi}^{f}(x),p)+D_{f}(p,Res_{\Theta,\varphi}^{f}(x))\\
&&\leq D_{f}(Res_{\Theta,\varphi}^{f}(x),p)-D_{f}(Res_{\Theta,\varphi}^{f}(x),x)+D_{f}(p,x)-D_{f}(p,p).
\end{eqnarray*}
Hence, 
$$D_{f}(p,Res_{\Theta,\varphi}^{f}(x))+D_{f}(Res_{\Theta,\varphi}^{f}(x),x)\leq D_{f}(p,x).$$
\end{proof}
\begin{lemma}\cite{xu}\label{444}
Assume that $\{x_{n}\}$ is a sequence of nonnegative real numbers such that 
$$x_{n+1}\leq (1-\alpha_{n})x_{n}+\beta_{n}, \ \ \ \ \forall n\geq1,$$
where $\{\alpha_{n}\}$ is a sequence in $(0,1)$ and $\{\beta_{n}\}$ is a sequence such that
\begin{enumerate}
\item
$\sum_{n=1}^{\infty}\alpha_{n}=+\infty;$

\item
$\limsup_{n\to\infty}\frac{\beta_{n}}{x_{n}}\leq 0$ or $\sum_{n=1}^{\infty}|\beta_{n}|<+\infty.$
\end{enumerate}
Then $\lim_{n\to\infty}x_{n}=0$.
\end{lemma}
\section{Main result}
\begin{theorem}\label{mt}
Let $E$ be a real reflexive Banach space, $C$ be a nonempty, closed and convex subset of $E$. Let $f:E\to \mathbb{R}$ be a coercive Legendre function which is bounded, uniformly Fr\'{e}chet differentiable and totally convex on bounded subsets of $E$. Let $T_{i}:C\to C$, for $i=1,2,\ldots, N,$ be a finite family of Bregman strongly nonexpansive mappings with respect to $f$ such that $F(T_{i})=\widehat{F}(T_{i})$ and each $T_{i}$ is uniformly continuous. Let $\Theta:C\times C\to\mathbb{R}$ satisfying conditions ($A_{1}$)-($A_{5}$) and $\left(\cap_{i=1}^{N}F(T_{i})\right)\cap MEP(\Theta)$ is nonempty and bounded. Let $\{x_{n}\}$ be a sequence generated by 
\begin{eqnarray}
x_{1}&=&x\in C \ \ \ \ \  \text{chosen arbitrarily},\nonumber\\
z_{n}&=&Res_{\Theta,\varphi}^{f}(x_{n}),\nonumber\\
y_{n}&=&proj_{C}^{f}\nabla f^{*}(\beta_{n}\nabla f(x_{n})+(1-\beta_{n})\nabla f(T(z_{n})))\nonumber\\
x_{n+1}&=&proj_{C}^{f}\nabla f^{*}(\alpha_{n}\nabla f(x_{n})+(1-\alpha_{n})\nabla f (T(y_{n}))),\label{main}
\end{eqnarray}
where $T=T_{N}\circ T_{N-1}\circ\ldots\circ T_{1}$, $\{\alpha_{n}\}, \{\beta_{n}\}\subset (0,1)$ satisfying $\lim_{n\to\infty}\alpha_{n}=0$ and $\sum_{n=1}^{\infty}\alpha_{n}=\infty$. Then $\{x_{n}\}$ converges strongly to $proj_{(\cap_{i=1}^{N}F(T_{i}))\cap MEP(\Theta)}x$.
\end{theorem}
\begin{proof}
We note from Lemma \ref{333} that $F(T_{i})$, for each $i\in\{1,2,\ldots,N\}$ is closed and convex and hence $\cap_{i=1}^{N}F(T_{i})$ is closed and convex. \\
 Let $p=proj_{(\cap_{i=1}^{N}F(T_{i}))\cap GMEP(\Theta)}x\in (\cap_{i=1}^{N}F(T_{i}))\cap GMEP(\Theta)$. Then $p\in (\cap_{i=1}^{N}F(T_{i}))$ and $p\in GMEP(\Theta)$. Now, by using (\ref{main}) and Lemma \ref{nv}, we have $D_{f}(p,z_{n})=D_{f}(p,Res_{\Theta,\varphi,\Psi}^{f}(x_{n}))\leq D_{f}(p,x_{n})$, so
\begin{eqnarray}
D_{f}(p,y_{n})&=&D_{f}(p,proj_{C}^{f}\nabla f^{*}(\beta_{n}\nabla f(x_{n})+(1-\beta_{n})\nabla f(T(z_{n}))))\nonumber\\
&\leq&D_{f}(p,\nabla f^{*}(\beta_{n}\nabla f(x_{n})+(1-\beta_{n})\nabla f(T(z_{n}))))\nonumber\\
&\leq &\beta_{n}D_{f}(p,x_{n})+(1-\beta_{n})D_{f}(p,T(z_{n}))\nonumber\\
&\leq&\beta_{n}D_{f}(p,x_{n})+(1-\beta_{n})D_{f}(p,z_{n})\nonumber\\
&\leq &\beta_{n}D_{f}(p,x_{n})+(1-\beta_{n})D_{f}(p,x_{n})\nonumber\\
&\leq & D_{f}(p,x_{n}).\label{3.1}
\end{eqnarray} 
 By (\ref{main}) and (\ref{3.1}), we have 
\begin{eqnarray*}
D_{f}(p,x_{n+1})&=&D_{f}(p,proj_{C}^{f}\nabla f^{*}(\alpha_{n}\nabla f(x_{n})+(1-\alpha_{n})\nabla f (T(y_{n}))))\\
&\leq&D_{f}(p,\nabla f^{*}(\alpha_{n}\nabla f(x_{n})+(1-\alpha_{n})\nabla f (T(y_{n}))))\\
&\leq& \alpha_{n}D_{f}(p,x_{n})+(1-\alpha_{n})D_{f}(p,T(y_{n}))\\
&\leq& \alpha_{n}D_{f}(p,x_{n})+(1-\alpha_{n})D_{f}(p,y_{n})\\
&\leq& \alpha_{n}D_{f}(p,x_{n})+(1-\alpha_{n})D_{f}(p,x_{n})\\
&\leq& D_{f}(p,x_{n}).
\end{eqnarray*}
Hence $\{D_{f}(p,x_{n})\}$ and $D_{f}(p,Ty_{n})$  are bounded. Moreover, by Lemma \ref{taz} we get that the sequences $\{x_{n}\}$ and $\{T(y_{n})\}$ are bounded. 

From the fact that $\alpha_{n}\to0$ as $n\to\infty$, Lemma \ref{niaz}   
 we get that
\begin{eqnarray*}
D_{f}(T(y_{n}),x_{n+1})&\leq& D_{f}(T(y_{n}),proj_{C}^{f}\nabla f^{*}(\alpha_{n}\nabla f(x_{n})+(1-\alpha_{n})\nabla f (T(y_{n})))\\
&\leq&D_{f}(T(y_{n}), \nabla f^{*}(\alpha_{n}\nabla f(x_{n})+(1-\alpha_{n})\nabla f (T(y_{n})))\\
&\leq& \alpha_{n}D_{f}(T(y_{n}), x_{n})+(1-\alpha_{n})D_{f}(T(y_{n}),T(y_{n}))\\
&=&\alpha_{n}D_{f}(T(y_{n}), x_{n})\\
&=&0.
\end{eqnarray*}
Therefore, by Lemma \ref{lem6}, we have
\begin{equation}\label{rd}
\|x_{n+1}-T(y_{n})\|\to0, \ \ \ as \ \ \ n\to\infty.
\end{equation}
On the other hand, by Lemma \ref{niaz}, we have
\begin{eqnarray*}
\lim_{n\to\infty} D_{f}(x_{n},z_{n})&=&\lim_{n\to\infty} D_{f}(x_{n},Res_{\Theta,\varphi}^{f}(x_{n})\\
&\leq& \lim_{n\to\infty}(D_{f}(p,Res_{\Theta,\varphi}^{f}(x_{n}))-D_{f}(p,x_{n}))\\
&\leq& \lim_{n\to\infty}(D_{f}(p,x_{n})-D_{f}(p,x_{n}))\\
&=&0.
\end{eqnarray*}
By Lemma \ref{lem6}, we obtain
\begin{equation}\label{3.8}
\lim_{n\to\infty}\|x_{n}-z_{n}\|=0.
\end{equation}
Since $f$ is  uniformly Fr\'{e}chet differentiable on bounded subsets of $E$, by Lemma \ref{lem7}, $\nabla f$ is norm-to-norm uniformly continuous on bounded subsets of $E$. So,
\begin{equation}\label{3.9}
\lim_{n\to\infty}\|\nabla f(x_{n})-\nabla f(z_{n})\|_{*}=0.
\end{equation}
Since $f$ is uniformly Fr\'{e}chet differentiable, it is also uniformly continuous, we get
\begin{equation}\label{3.10}
\lim_{n\to\infty}\|f(x_{n})-f(z_{n})\|=0.
\end{equation}
By Bregman distance we have
\begin{eqnarray*}
&&D_{f}(p,x_{n})-D_{f}(p,z_{n})\\
&&=f(p)-f(x_{n})-\langle \nabla f(x_{n}),p-x_{n}\rangle -f(p)+f(z_{n})+\langle \nabla f(z_{n}),p-z_{n}\rangle\\
&&=f(z_{n})-f(x_{n})+\langle \nabla f(z_{n}),p-z_{n}\rangle-\langle \nabla f(x_{n}),p-x_{n}\rangle\\
&&=f(z_{n})-f(x_{n})+\langle \nabla f(z_{n}),x_{n}-z_{n}\rangle-\langle \nabla f(z_{n})-\nabla f(x_{n}),p-x_{n}\rangle,
\end{eqnarray*}
for each $p\in \cap_{i=1}^{N}F(T_{i})$. By (\ref{3.8})-(\ref{3.10}), we obtain
\begin{equation}\label{3.11}
\lim_{n\to\infty}(D_{f}(p,x_{n})-D_{f}(p,z_{n}))=0.
\end{equation}
By above equation, we have
\begin{eqnarray*}
D_{f}(z_{n},y_{n})&=&D_{f}(p,y_{n})-D_{f}(p,z_{n})\\
&=&D_{f}(p,proj_{C}^{f}\nabla f^{*}(\alpha_{n}\nabla f(x_{n})+(1-\alpha_{n})\nabla f(T(z_{n}))-D_{f}(p,z_{n}))\\
&\leq&D_{f}(p,\nabla f^{*}(\alpha_{n}\nabla f(x_{n})+(1-\alpha_{n})\nabla f(T(z_{n}))-D_{f}(p,z_{n}))\\
&\leq& \alpha_{n}D_{f}(p,x_{n})+(1-\alpha_{n})D_{f}(p,T(z_{n})-D_{f}(p,z_{n})\\
&\leq&\alpha_{n}D_{f}(p,x_{n})+(1-\alpha_{n})D_{f}(p,z_{n})-D_{f}(p,z_{n})\\
&=&\alpha_{n}(D_{f}(p,x_{n})-D_{f}(p,z_{n}))\\
&=&0.
\end{eqnarray*}
By (\ref{3.11}), we have
\begin{equation}\label{3.13}
\lim_{n\to\infty}\|z_{n}-y_{n}\|=0.
\end{equation}
Note that
$$\|x_{n}-y_{n}\|\leq \|x_{n}-z_{n}\|+\|z_{n}-y_{n}\|.$$
By applying (\ref{3.8}) and (\ref{3.13}), we can write

\begin{equation}\label{15}
\lim_{n\to\infty}\|x_{n}-y_{n}\|=0.
\end{equation}
Now, we claim that 
\begin{equation}\label{22}
\lim_{n\to\infty}\|x_{n}-Tx_{n}\|=0.
\end{equation}
Since $f$ is  uniformly Fr\'{e}chet differentiable on bounded subsets of $E$, by Lemma \ref{lem7}, $\nabla f$ is norm-to-norm uniformly continuous on bounded subsets of $E$. So,
\begin{equation}\label{16}
\lim_{n\to\infty}\|\nabla f(x_{n})-\nabla f(y_{n})\|_{*}=0.
\end{equation}

Since $f$ is uniformly Fr\'{e}chet differentiable, it is also uniformly continuous, we get
\begin{equation}\label{17}
\lim_{n\to\infty}\|f(x_{n})-f(y_{n})\|=0.
\end{equation}
By Bregman distance we have
\begin{eqnarray*}
&&D_{f}(p,x_{n})-D_{f}(p,y_{n})\\
&&=f(p)-f(x_{n})-\langle \nabla f(x_{n}),p-x_{n}\rangle -f(p)+f(y_{n})+\langle \nabla f(y_{n}),p-y_{n}\rangle\\
&&=f(y_{n})-f(x_{n})+\langle \nabla f(y_{n}),p-y_{n}\rangle-\langle \nabla f(x_{n}),p-x_{n}\rangle\\
&&=f(y_{n})-f(x_{n})+\langle \nabla f(y_{n}),x_{n}-y_{n}\rangle-\langle \nabla f(y_{n})-\nabla f(x_{n}),p-x_{n}\rangle,
\end{eqnarray*}
for each $p\in \cap_{i=1}^{N}F(T_{i})$. By (\ref{15})-(\ref{17}), we obtain
\begin{equation}\label{18}
\lim_{n\to\infty}(D_{f}(p,x_{n})-D_{f}(p,y_{n}))=0.
\end{equation}
By above equation, we have
\begin{eqnarray*}
D_{f}(y_{n},x_{n+1})&=&D_{f}(p,x_{n+1})-D_{f}(p,y_{n})\\
&=&D_{f}(p,proj_{C}^{f}\nabla f^{*}(\alpha_{n}\nabla f(x_{n})+(1-\alpha_{n})\nabla f(T(x_{n}))-D_{f}(p,y_{n}))\\
&\leq&D_{f}(p,\nabla f^{*}(\alpha_{n}\nabla f(x_{n})+(1-\alpha_{n})\nabla f(T(x_{n}))-D_{f}(p,y_{n}))\\
&\leq& \alpha_{n}D_{f}(p,x_{n})+(1-\alpha_{n})D_{f}(p,T(y_{n})-D_{f}(p,y_{n})\\
&\leq&\alpha_{n}D_{f}(p,x_{n})+(1-\alpha_{n})D_{f}(p,y_{n})-D_{f}(p,y_{n})\\
&=&\alpha_{n}(D_{f}(p,x_{n})-D_{f}(p,y_{n}))\\
&=&0.
\end{eqnarray*}
By Lemma \ref{lem6}, we have
$$\lim_{n\to\infty}\|y_{n}-x_{n+1}\|=0.$$
From above equation and (\ref{rd}), we can write
\begin{eqnarray}
\|y_{n}-T(y_{n})\|&\leq&\|y_{n}-x_{n+1}\|+\|x_{n+1}-T(y_{n})\|\nonumber\\
&=&0\label{41}
\end{eqnarray}
when $n\to\infty$.
By applying the triangle inequality, we get
$$\|x_{n}-T(x_{n})\|\leq \|x_{n}-y_{n}\|+\|y_{n}-T(y_{n})\|+\|T(y_{n})-T(x_{n})\|.$$
By (\ref{15}), (\ref{41}) and since $T_{i}$ is uniformly continuous for each $i\in\{1,2,\ldots,N\}$ we have
$$\lim_{n\to\infty}\|x_{n}-T(x_{n})\|=0.$$ 
As claimed in (\ref{22}).\\
Since $\|x_{n_{k}}-T(x_{n_{k}})\|\to0$ as $k\to\infty$, we have $q\in \cap_{i=1}^{N}F(T_{i})$.\\
From (\ref{3.8}) we can write
$$\lim_{n\to\infty}\|Jz_{n}-Jx_{n}\|=0.$$
Here, we prove that $q\in MEP(\Theta)$. For this reason, consider that $z_{n}=Res_{\Theta,\varphi}^{f}(x_{n})$, so we have
$$\Theta(z_{n},z)+\varphi(z)+\langle Jz_{n}-Jx_{n},z-z_{n}\rangle\geq \varphi(z_{n}), \ \ \forall z\in C.$$
From ($A_{2}$), we have
$$\Theta (z,z_{n})\leq -\Theta (z_{n},z)\leq \varphi(z)-\varphi(z_{n})+\langle Jz_{n}-Jx_{n},z-z_{n}\rangle, \ \ \forall z\in C.$$
Hence, 
$$\Theta(z,z_{n_{i}})\leq \varphi(z)- \varphi(z_{n_{i}})+\langle Jz_{n_{i}}-Jx_{n_{i}},z-z_{n_{i}}\rangle, \ \ \forall z\in C.$$
Since $z_{n_{i}}\rightharpoonup q$ and from the weak lower semicontinuity of $\varphi$ and $\Theta (x,y)$ in the second variable $y$, we also have
$$\Theta (z,q)+\varphi(q)-\varphi(z)\leq0, \ \ \ \forall z\in C.$$
For $t$ with $0\leq t\leq 1$ and $z\in C$, let $z_{t}=tz+(1-t)q$. Since $z\in C$ and $q\in C$ we have $z_{t}\in C$ and hence $\Theta(z_{t},q)+\varphi(q)-\varphi(z_{t})\leq 0$. So, from the continuity of the equilibrium bifunction $\Theta(x,y)$ in the second variable $y$, we have
\begin{eqnarray*}
0&=&\Theta(z_{t},z_{t})+\varphi(z_{t})-\varphi(z_{t})\\
&\leq& t\Theta(z_{t},z)+(1-t)\Theta(z_{t},q)+t\varphi(z)+(1-t)\varphi(q)-\varphi(z_{t})\\
&\leq& t[\Theta(z_{t},z)+\varphi(z)-\varphi(z_{t})].
\end{eqnarray*}
 Therefore, $\Theta(z_{t},z)+\varphi(z)-\varphi(z_{t})\geq0$. Then, we have 
 $$\Theta(q,z)+\varphi(z)-\varphi(q)\geq0, \ \ \ \forall y\in C.$$ 
 Hence we have $q\in MEP(\Theta)$. We showed that $q\in(\cap_{i=1}^{N}F(T_{i}))\cap MEP(\Theta)$.\\
 Since $E$ is reflexive and $\{x_{n}\}$ is bounded, there exists a subsequence $\{x_{n_{k}}\}$ of $\{x_{n}\}$ such that $\{x_{n_{k}}\}\rightharpoonup q\in C$ and
 $$\limsup_{n\to\infty}\langle \nabla f(x_{n})-\nabla f(p),x_{n}-p\rangle=\langle \nabla f(x_{n})-\nabla f(p),q-p\rangle.$$
 On the other hand, since $\|x_{n_{k}}-Tx_{n_{k}}\|\to0$ as $k\to\infty$, we have $q\in \cap_{i=1}^{N}F(T_{i})$. It follows from the definition of the Bregman projection that
 \begin{equation}\label{60}
 \limsup_{n\to\infty}\langle \nabla f(x_{n})-\nabla f(p),x_{n}-p\rangle=\langle \nabla f(x_{n})-\nabla f(p),q-p\rangle\leq 0.
 \end{equation}
From (\ref{29}), we obtain
\begin{eqnarray*}
D_{f}(p,x_{n+1})&=&D_{f}(p,proj_{C}^{f}\nabla f^{*}(\alpha_{n}\nabla f(x_{n})+(1-\alpha_{n})\nabla f(T(x_{n}))))\\
&\leq&D_{f}(p,\nabla f^{*}(\alpha_{n}\nabla f(x_{n})+(1-\alpha_{n})\nabla f(T(x_{n}))))\\
&=
&V_{f}(p,\alpha_{n}\nabla f(x_{n})+(1-\alpha_{n})\nabla f(T(y_{n})))\\
&\leq& V_{f}(p,\alpha_{n}\nabla f(x_{n})+(1-\alpha_{n})\nabla f(T(y_{n}))-\alpha_{n}(\nabla f(x_{n})-\nabla f(p)))\\
&&+\langle \alpha_{n}(\nabla f(x_{n})-\nabla f(p)),x_{n+1}-p\rangle\\
&=&V_{f}(p,\alpha_{n}\nabla f(p)+(1-\alpha_{n})\nabla f(T(y_{n})\\
&&+ \alpha_{n}\langle \nabla f(x_{n})-\nabla f(p),x_{n+1}-p\rangle\\
&\leq&\alpha_{n} V_{f}(p,\nabla f(p))+(1-\alpha_{n})V_{f}(p,\nabla f(T(y_{n})))\\
&&+ \alpha_{n}\langle \nabla f(x_{n})-\nabla f(p),x_{n+1}-p\rangle\\
&=&(1-\alpha_{n})D_{f}(p,T(y_{n})+\alpha_{n}\langle \nabla f(x_{n})-\nabla f(p),x_{n+1}-p\rangle\\
&\leq&(1-\alpha_{n})D_{f}(p,x_{n})+\alpha_{n}\langle\nabla f(x_{n})-\nabla f(p),x_{n+1}-p\rangle.
\end{eqnarray*}
By Lemma \ref{444} and (\ref{60}), we can conclude that $\lim_{n\to\infty}D_{f}(p,x_{n})=0$. Therefore, by Lemma \ref{lem6}, $x_{n}\to p$. This completes the proof.
\end{proof}
Let $\beta_{n}=0$, \ $\forall n\in\mathbb{N}$ in Theorem \ref{main}, we have a generalization of H. zegeye result \cite{zeg}.

If in Theorem \ref{mt}, we consider a single Bregman strongly nonexpansive mapping, we have the following corollary.
\begin{corollary}
Let $E$ be a real reflexive Banach space, $C$ be a nonempty, closed and convex subset of $E$. Let $f:E\to \mathbb{R}$ be a coercive Legendre function which is bounded, uniformly Fr\'{e}chet differentiable and totally convex on bounded subsets of $E$. Let $T$ be a Bregman strongly nonexpansive mappings with respect to $f$ such that $F(T)=\widehat{F}(T)$ and $T$ is uniformly continuous. Let $\Theta:C\times C\to\mathbb{R}$ satisfying conditions ($A_{1}$)-($A_{5}$) and $F(T)\cap MEP(\Theta)$ is nonempty and bounded. Let $\{x_{n}\}$ be a sequence generated by 
\begin{eqnarray*}
x_{1}&=&x\in C \ \ \ \ \  \text{chosen arbitrarily},\nonumber\\
z_{n}&=&Res_{\Theta,\varphi}^{f}(x_{n}),\nonumber\\
y_{n}&=&proj_{C}^{f}\nabla f^{*}(\beta_{n}\nabla f(x_{n})+(1-\beta_{n})\nabla f(T(z_{n})))\nonumber\\
x_{n+1}&=&proj_{C}^{f}\nabla f^{*}(\alpha_{n}\nabla f(x_{n})+(1-\alpha_{n})\nabla f (T(y_{n}))),
\end{eqnarray*}
where $\{\alpha_{n}\},\{\beta_{n}\}\subset (0,1)$ satisfying $\lim_{n\to\infty}\alpha_{n}=0$ and $\sum_{n=1}^{\infty}\alpha_{n}=\infty$. Then $\{x_{n}\}$ converges strongly to $proj_{F(T)\cap MEP(\Theta)}x$.
\end{corollary}

If in Theorem \ref{mt}, we assume that $E$ is a uniformly smooth and uniformly convex Banach space and $f(x):=\frac{1}{p}\|x\|^{p} \ \ (1<p<\infty)$, we have that $\nabla f=J_{p}$, where $J_{p}$ is the generalization duality mapping from $E$ onto $E^{*}$. Thus, we get the following corollary.
\begin{corollary}
Let $E$ be a uniformly smooth and uniformly convex Banach space and $f(x):=\frac{1}{p}\|x\|^{p} \ \ (1<p<\infty)$. Let $C$ be a nonempty, closed and convex subset of $\text{int}(\text{dom} f)$ and $T_{i}:C\to C$, for $i=1,2,\ldots, N,$ be a finite family of Bregman strongly nonexpansive mappings with respect to $f$ such that $F(T_{i})=\widehat{F}(T_{i})$and each $T_{i}$ is uniformly continuous. Let $\Theta:C\times C\to\mathbb{R}$ satisfying conditions ($A_{1}$)-($A_{5}$) and $\left(\cap_{i=1}^{N}F(T_{i})\right)\cap MEP(\Theta)$ is nonempty and bounded. Let $\{x_{n}\}$ be a sequence generated by 
\begin{eqnarray*}
x_{1}&=&x\in C \ \ \ \ \  \text{chosen arbitrarily},\nonumber\\
z_{n}&=&Res_{\Theta,\varphi}^{f}(x_{n}),\nonumber\\
y_{n}&=&proj_{C}^{f}J_{p}^{-1}(\beta_{n}J_{p} f(x_{n})+(1-\beta_{n})J_{p}(T(z_{n})))\nonumber\\
x_{n+1}&=&proj_{C}^{f}J_{p}^{-1}(\alpha_{n}J_{p}(x_{n})+(1-\alpha_{n})J_{p} (T(y_{n}))),
\end{eqnarray*}
where $T=T_{N}\circ T_{N-1}\circ\ldots\circ T_{1}$, $\{\alpha_{n}\}, \{\beta_{n}\}\subset (0,1)$ satisfying $\lim_{n\to\infty}\alpha_{n}=0$ and $\sum_{n=1}^{\infty}\alpha_{n}=\infty$. Then $\{x_{n}\}$ converges strongly to $proj_{(\cap_{i=1}^{N}F(T_{i}))\cap MEP(\Theta)}x$.\\

\end{corollary}

% Non-BibTeX users please use


\begin{thebibliography}{99}
\bibitem{alb}
Y.I. Alber, \textit{Metric and generalized projection operators in Banach spaces: properties and applications, in: A.G. Kartsatos (Ed.), Theory and Applications of Nonlinear Operator of Accretive and Monotone Type}, Marcel Dekker, New York, (1996) 15--50.

\bibitem{blu}
E. Blum, W. Oettli, \textit{From optimization and variational inequalities to equilibrium problems}, Math. Student \textbf{63} (1994) 123--145.

\bibitem{asl2}
M. Aslam Noor, \textit{Generalized mixed quasi-equilibrium problems with trifunction}, Appl. Math. Lett. \textbf{18} (2005) 695--700.

\bibitem{asl}
M. Aslam Noor, W. Oettli, \textit{On general nonlinear complementarity problems and quasi equilibria}, Matematiche (Catania) \textbf{49} (1994) 313--331.

\bibitem{bau}
H. H. Bauschke, J. M. Borwein, P. L. Combettes, \textit{Essential smoothness, essential strict convexity, and
Legendre functions in Banach spaces}, Commun. Contemp. Math. \textbf{3} (2001) 615--647.


\bibitem{bon}
J. F. Bonnans, A. Shapiro, \textit{Perturbation analysis of optimization problem}. NewYork (NY) Springer, 2000.
\bibitem{bru}
R.E. Bruck, S. Reich, \textit{Nonexpansive projections and resolvents of accretive operators in Banach spaces}, Houston J. Math. \textbf{3} (1977)
459--470.
\bibitem{but}
D. Butnariu, E. Resmerita, \textit{Bregman distances, totally convex functions and a method for solving operator equations in Banach spaces}, Abstr. Appl.
Anal. Art. ID 84919 (2006) 1--39. 
\bibitem{but2}
D. Butnariu, A. N. Iusem, \textit{ Totally Convex Functions for Fixed Points Computation and Infinite Dimensional
Optimization}, Applied Optimization, \textbf{40}  Kluwer Academic, Dordrecht (2000).

\bibitem{cen}
L.C. Ceng, J.C. Yao, \textit{A hybrid iterative scheme for mixed equilibrium problems and fixed point problems}, J. Comput. Appl. Math. \textbf{214} (2008) 186--201.
\bibitem{cens}
Y. Censor, A. Lent, \textit{An iterative row-action method for interval convex programming}, J. Optim. Theory Appl. \textbf{34} (1981) 321--353.
\bibitem{cha}
O. Chadli, N.C. Wong, J.C. Yao, \textit{Equilibrium problems with applications to eigenvalue problems}, J. Optim. Theory Appl. \textbf{117} (2003) 245--266.
\bibitem{cha2}
O. Chadli, S. Schaible, J.C. Yao, \textit{Regularized equilibrium problems with an application to noncoercive hemivariational inequalities}, J. Optim. Theory
Appl. \textbf{121} (2004) 571--596.
\bibitem{hir}
J. B. Hiriart-Urruty, C. Lemar\'{e}chal, \textit{Grundlehren der mathematischen Wissenschaften, in: Convex Analysis and Minimization Algorithms II}, \textbf{ 306},
Springer-Verlag, (1993).

\bibitem{kas}
G. Kassay,  S. Reich, S. Sabach, \textit{Iterative methods for solving systems of variational inequalities
in reflexive Banach spaces}, SIAM J. Optim.  \textbf{21}  (2011) 1319--1344.


\bibitem{koh}
F. Kohsaka, W. Takahashi, \textit{Proximal point algorithms with Bregman functions in Banach spaces}, J. Nonlinear Convex Anal. \textbf{6} (2005) 505--523.

\bibitem{kum}
W. Kumama, U. Witthayaratb, P. Kumam,
S. Suantaie, K. Wattanawitoon, \textit{
Convergence theorem for equilibrium problem and Bregman strongly
nonexpansive mappings in Banach spaces}, Optimization: A Journal of Mathematical
Programming and Operations Research, DOI: 10.1080/02331934.2015.1020942.


\bibitem{kon}
I.V. Konnov, S. Schaible, J.C. Yao, \textit{Combined relaxation method for mixed equilibrium problems}, J. Optim. Theory Appl. \textbf{126} (2005) 309--322.
\bibitem{mar}
V. Martn-Marquez, S. Reich, S. Sabach, \textit{Iterative methods for approximating fixed points of
Bregman nonexpansive operators}, Discrete Contin. Dyn. Syst. Ser. S. \textbf{6} (2013) 1043--1063.


\bibitem{mor}
J. J. Moreau, \textit{  Sur la fonction polaire d’une fonction semi-continue supérieurement [On the polar
function of a semi-continuous function superiorly]},  C. R. Acad. Sci. Paris. \textbf{258} (1964) 1128--1130.
\bibitem{pen}
J. W. Peng, J. C. Yao, \textit{Strong convergence theorems of iterative scheme based on the
extragradient method for mixed equilibrium problems and fixed
point problems}, Math. Comp. Model. \textbf{49} (2009) 1816--1828.


\bibitem{phe}

R. P. Phelps, \textit{Convex Functions, Monotone Operators, and Differentiability}, second ed., in: Lecture Notes in Mathematics, vol.
1364, Springer Verlag, Berlin, 1993.
\bibitem{rei2}
S. Reich, \textit{A weak convergence theorem for the alternating method with Bregman distances, in: Theory and Applications of
Nonlinear Operators of Accretive and Monotone Type}, Marcel Dekker, New York, (1996) 313--318.
\bibitem{rei4}
S. Reich, S. Sabach, \textit{A strong convergence theorem for a proximal-type algorithm in reflexive Banach spaces},
J. Nonlinear Convex Anal. \textbf{10} (2009) 471-485.
\bibitem{rei3}
S. Reich, S. Sabach, \textit{Existence and approximation of fixed points of Bregman firmly nonexpansive mappings in
reflexive Banach spaces. In: Fixed-Point Algorithms for Inverse Problems in Science and Engineering}, Optimization and Its Applications, 
\textbf{49} (2011) 301--316.

\bibitem{rei}
S. Reich, S. Sabach, \textit{Two strong convergence theorems for a proximal method in reflexive Banach spaces}, Numer. Funct. Anal.
Optim. \textbf{31} (2010) 22--44.


\bibitem{roc}
R. T. Rockafellar, \textit{Level sets and continuity of conjugate convex functions}, Trans. Amer. Math.
Soc. \textbf{123} (1966) 46--63.
\bibitem{sua}
S. Suantai, Y. J. Choc, P. Cholamjiak, \textit{ Halpern’s iteration for Bregman strongly nonexpansive mappings in reflexive Banach spaces},
Computers and Mathematics with Applications, \textbf{64} (2012) 489--499.

\bibitem{xu}
H. K. Xu, \textit{An iterative approach to quadratic optimization}, J. Optim. Theory Appl. \textbf{116} (2003) 659--678.

\bibitem{yao}
Y. Yao, M. Aslam Noor, S. Zainab, Y.C Liou, \textit{Mixed equilibrium problems and Optimization problems}
J. Math. Anal. Appl. \textbf{354} (2009) 319--329 .


\bibitem{zal}
C. Z\'{a}linescu,  \textit{Convex analysis in general vector spaces}, River Edge (NJ), World Scientific (2002).
\bibitem{zeg}
H. Zegeye, \textit{Convergence theorems for Bregman strongly nonexpansive
mappings in reflexive Banach spaces}, Filomat. \textbf{7} (2014) 1525--1536.
\end{thebibliography}
\end{document}